\documentclass[a4paper,12pt,onesided]{amsart}

\usepackage{mystyle_english_article_2016}
\usepackage{cleveref}

\usepackage{anysize}

\newcommand*{\be}{\begin{equation}}
\newcommand*{\ee}{\end{equation}}
\newcommand*{\ba}{\begin{aligned}}
\newcommand*{\ea}{\end{aligned}}
\newcommand*{\barr}{\begin{array}{c}}
\newcommand*{\earr}{\end{array}}

\usepackage{url}

\begin{document}

\title[Singularity of self-similar measures]
{Singularity versus exact overlaps for self-similar measures}

\author{K\'{a}roly  Simon}
\address{K\'{a}roly Simon, Institute of Mathematics, Technical
University of Budapest, H-1529 B.O.box 91, Hungary
\newline
\tt{simonk{@}math.bme.hu}}

\author{Lajos V\'{a}g\'{o}}
\address{Lajos V\'{a}g\'{o}, Institute of Mathematics, Technical
University of Budapest, H-1529 B.O.box 91, Hungary. MTA-BME Stochastic Research Group (temporary affiliation)
\newline
\tt{vagolala{@}math.bme.hu}}

 \thanks{
\indent The research of Simon and Vago was supported by OTKA Foundation
{\#}K71693 .  This work was partially supported by the grant 346300 for IMPAN from the Simons
Foundation.}

\maketitle

\begin{abstract}
In this note we present some  one-parameter families of homogeneous self-similar measures on the line such that
 \begin{itemize}
   \item the  similarity dimension is greater than $1$ for all parameters and
   \item the singularity of some of the self-similar measures from this family  is not caused by exact overlaps between the cylinders.
 \end{itemize}
 We can obtain such a family as the angle-$\alpha$ projections of the natural measure of the Sierpi\'nski carpet. We present more general one-parameter families of self-similar measures $\nu_\alpha$, such that the set of parameters $\alpha$ for which $\nu_\alpha$ is singular is a dense $G_\delta$ set but this "exceptional" set of parameters of singularity  has  zero Hausdorff dimension.
\end{abstract}

\section{Introduction}\label{S72}
\subsection{Description of the problem investigated}
In recent years there has been a rapid development in the field of self-similar Iterated Function Systems  (IFS) with overlapping construction. Most importantly, Hochman \cite{Hochman} proved for any self-similar measure $\nu$  that we can have dimension drop (that is $\dim_{\rm H} \nu<\min\left\{1,\dim_{\rm S}\nu \right\}$, )
only if there is a superexponential concentration of cylinders (see Section \ref{S118} for the definitions of the various dimensions used in the paper).  Consequently,
for a one-parameter family of self-similar measures $\left\{\nu_\alpha\right\}_\alpha$ on $\mathbb{R}$, satisfying a certain non-degeneracy condition (Definition \ref{S80})
the Hausdorff dimension of the measure $\nu_\alpha$ is equal to the minimum of its similarity dimension and $1$  for all parameters $\alpha$ except for a small exceptional set of parameters $E$. This exceptional set $E$ is so small that its packing dimension (and consequently its Hausdorff dimension) is zero.
 The corresponding problem for the singularity versus  absolute continuity of self-similar  measures
was treated by Shmerkin and Solomyak \cite{shmerkin2014absolute}. They considered one-parameter families of self-similar measures constructed by one-parameter families of homogeneous self-similar IFS, also satisfying the non-degeneracy condition of Hochman Theorem. It was proved in
\cite[Theorem A]{shmerkin2014absolute}
that for such families $\left\{\nu_\alpha\right\}$ of self-similar measures  if the similarity dimension of the measures in the family is greater than $1$
then for all but a set of Hausdorff dimension zero of parameters $\alpha$, the measure $\nu_\alpha$ is absolute continuous with respect to the Lebesgue measure. The results presented in this note imply that in this case it can happen that the set of exceptional parameters have packing dimension $1$ as opposed  to Hochman's Theorem where we remind that  the packing dimension of the set of exceptional parameters is equal to $0$.

Still,  we do not know what causes the drop of dimension or the singularity of a self-similar measure on the line of similarity dimension greater than $1$. In particular it is a natural question whether the only reason for the drop of the dimension or singularity of self-similar measures having similarity dimension  larger than $1$ is the "exact overlap". More precisely, let $\left\{\varphi_i\right\}_{i=1}^{m}$ be a self-similar IFS and $\nu$ be a corresponding self-similar measure. We say that there is an exact overlap if we can find two distinct $\mathbf{i}=(i_1, \dots ,i_k)$ and $\mathbf{j}=(j_1, \dots ,j_\ell )$ finite words such that
      \begin{equation}\label{S54}
        \varphi_{i_1}\circ \cdots \circ \varphi_{i_k}
        =
        \varphi_{j_1}\circ \cdots \circ \varphi_{j_\ell }.
      \end{equation}

The following two questions have naturally arisen for a  long time (e.g. Question 1  below appeared as \cite[Question 2.6]{peres2000problems}):

\begin{description}
  \item[Question 1] Is it true that a self-similar measure has Hausdorff dimension strictly smaller than the minimum of $1$ and its similarity dimension only if we have exact overlap?
  \item[Question 2] Is it true  for a self-similar measure $\nu$ having similarity dimension greater than one, that $\nu$ is singular only if there is exact overlap?
\end{description}
Most of the experts believe that the answer to Question 1 is positive and it has been confirmed in some special cases \cite{Hochman}. On the other hand, a result of
Nazarov, Peres and Shmerkin
  indicated that the answer to Question 2 should be negative.
   Namely, they
   constructed  in \cite{nazarov2012convolutions} a planar self-affine set having dimension greater than one, such that  the angle-$\alpha$ projection  of its natural measure
   was singular for a dense $G_\delta$ set of parameters  $\alpha$. However, this was not a family of self-similar measures. Up to our best knowledge
before this note, Question 2 has not yet been answered.
    \bigskip

\subsection{New results}
We consider one-parameter families of homogeneous self-similar measures on the line, having similarity dimension greater than $1$. We call the set of those parameters for which the measure is singular, set of parameters of singularity.
\begin{description}
  \item[(a)] We point out that the answer to Question 2 above is negative.  (Theorem \ref{S65}).
  \item[(b)] We consider one-parameter families of self-similar measures for which the set of parameters of singularity is big in the sense that it is a dense $G_\delta$ set but in the same time  the parameter set of singularity is small in the sense that it is a set of Hausdorff dimension zero.
We call such families antagonistic. We point out that there are many antagonistic families. Actually, we show that such antagonistic families are dense in a natural collection of one parameter families. (Proposition \ref{S102}.)
  \item[(c)] As a corollary, we obtain that it happens quite frequently that in Shmerkin-Solomyak Theorem (Theorem \ref{S85}) the exceptional set has packing dimension $1$. (Corollary \ref{S103}.)
  \item[(d)] We extend the scope of  \cite[Proposition 8.1]{peres2000sixty} from infinite Bernoulli convolution measures to
      very general one-parameter families of (not necessarily self-similar, or self-affine) IFS, and state that the parameter set of singularity is a $G_\delta$ set (Theorems \ref{S128}, \ref{S11}).
\end{description}

\subsection{Comments}
\bigskip
 The main goal of this note is to make the observation  that the  combination of an already existing method of Peres, Schlag and Solomyak \cite{peres2000sixty} and a result  due to Manning and  the first author of this note \cite{manning2013dimension} yields that the answer to Question 2 is negative.

  There are two ingredients of our argument:
   \begin{description}
     \item[(i)] The fact that the set of parameters of singularity is a $G_\delta$ set in any reasonable one-parameter family of self-similar measures on the line.
     \item[(ii)] The existence of a one-parameter family of self-similar measures having similarity dimension greater than one (for all parameters) with a dense set of parameters of singularity.
   \end{description}
    It turned out that both of these ingredients  have been available for a while in the literature. Although in an earlier version of this note the authors had their longer proof for \textbf{(i)}, we learned from
    B. Solomyak  that \textbf{(i)} has already been proved in
    \cite[Proposition 8.1]{peres2000sixty}
    in the
special  case of infinite Bernoulli convolutions.
  Actually, the authors of \cite{peres2000sixty} acknowledged that the  short and elegant proof of
  \cite[Proposition 8.1]{peres2000sixty} is due to
   Elon Lindenstrauss.
 We extend the scope of
   \cite[Proposition 8.1]{peres2000sixty} to a more general case. Then following the supposition of the anonymous referee we finally got a very general case.
	So, to prove \textbf{(i)}, we will present here a more detailed and very general extension of the proof of \cite[Proposition 8.1]{peres2000sixty}.

   On the other hand \textbf{(ii)} was proved in \cite{manning2013dimension}.

\subsection{Notation}\label{S99}
First we introduce the Hausdorff and similarity dimensions of a measure and then we present some definitions related to the singularity and absolute continuity of the family of measures considered in the paper.
\subsubsection{The different notions of dimensions used in the paper }\label{S118}
\begin{itemize}
  \item  The notion of the  \emph{Hausdorff and box dimension of a set }  is  well known (see e.g. \cite{FalconerTechniques}).
       \item  \emph{Hausdorff dimension of a measure}: Let $\mathfrak{m}$ be  a measure on $\mathbb{R}^d$. The Hausdorff dimension of $\mathfrak{m}$ is defined by
\begin{equation}\label{S117}
  \dim_{\rm H} \mathfrak{m}:=\inf\left\{
  \dim_{\rm H} A:\mathfrak{m}(A)>0, \mbox{ and $A$ is a Borel set}
  \right\},
\end{equation}
 see \cite[p. 170]{FalconerTechniques} for an equivalent definition.
  \item
We will use the following definition of the Packing dimension of a set $H\subset\mathbb{R}^d$ \cite[p. 23.]{FalconerTechniques}:
\begin{equation}\label{S200}
  \dim_{\rm P} H = \inf\{\sup_i \overline{\dim_{\rm B}} E_i\ :\ H\subset \bigcup_{i=1}^\infty E_i\},
\end{equation}
where $\overline{\dim_{\rm B}}$ stands for the upper box dimension.
The most important properties of the packing dimension can be found in
\cite{FalconerTechniques}.

  \item \emph{Similarity dimension of a self-similar measure}: Consider the self-similar IFS on the line: $\mathcal{F}:=\left\{\varphi(x):=r_i \cdot x+t_i\right\}_{i=1}^{m}$, where $r_i\in(-1,1)\setminus \left\{0\right\}$. Further we are given the probability vector  $\mathbf{w}:=(w_1, \dots ,w_m)$. Then there exists a unique measure $\nu$ satisfying $\nu(H)=\sum\limits_{i=1}^{m}w_i \cdot \nu\left(\varphi_{i}^{-1}
  \left(H\right)\right)$. (See \cite{FalconerTechniques}.) We call $\nu=\nu_{\mathcal{F},\mathbf{w}}$ the self-similar measure corresponding to $\mathcal{F}$ and $\mathbf{w}$. The similarity dimension of $\nu$ is defined by
  \begin{equation}\label{S71}
 \dim_{\rm S}(\nu_{\mathcal{F},\mathbf{w}}):=  \frac{\sum\limits_{i=1}^{m}w_i\log w_i}{\sum\limits_{i=1}^{m}w_i\log r_{ i}}.
  \end{equation}

\end{itemize}

 \subsubsection{The projected families of a self-similar measure}\label{S119}
Let
 \begin{equation}\label{S2}
   \mathcal{F}_\alpha:=\left\{
   \varphi_{i}^{\alpha}(x):=r_{\alpha,i} \cdot x+t_{i}^{(\alpha)}
   \right\}_{i=1}^{m},\quad \alpha\in A,
 \end{equation}
be a one-parameter family of self-similar  IFS on $\mathbb{R}$ and let $\mu$ be a measure on the symbolic space
$\Sigma:=\left\{1, \dots ,m\right\}^\mathbb{N}.
$
We write
$$
\varphi_{i_1 \dots i_n}^{\alpha}:=\varphi_{i_1}^{\alpha}\circ \cdots\circ\varphi_{i_n}^{\alpha} \mbox{ and }
r_{\alpha,i_1 \dots i_n}:=r_{\alpha,i_1} \cdots r_{\alpha,i_n}.
$$
The natural projection $\Pi_\alpha:\Sigma\to\mathbb{R}$
is defined by
\begin{equation}\label{S31}
  \Pi_\alpha(\mathbf{i}):=\lim\limits_{n\to\infty} \varphi_{i_1 \dots i_n}^{\alpha}(0)=
  \sum\limits_{k=1}^{\infty }t_{i_k}^{(\alpha)}r_{\alpha,i_1 \dots ,i_{k-1}},
\end{equation}
where $r_{\alpha,i_1 \dots ,i_{k-1}}:=1$ when $k=1$.
Let $\mu$ be a   probability measure on $\Sigma$. We study the family of its push forward measures $\left\{\nu_\alpha\right\}_{\alpha\in A}$:
\begin{equation}\label{S57}
  \nu_\alpha(H):=
  (\Pi_\alpha)_*\mu(H):=\mu(\Pi_\alpha^{-1}(H)),
\end{equation}
where $H$ is a Borel subset of $\Sigma$.

 The elements of the symbolic space $\Sigma:=\left\{1, \dots ,m\right\}^\mathbb{N}$ are denoted by $\mathbf{i}=(i_1,i_2, \dots )$.

 If $\mathbf{w}:=(w_1, \dots ,w_m)$ is a probability vector and $\mu$ is the infinite product of $\mathbf{w}$, that is
$\mu=\left\{w_1, \dots ,w_m\right\}^\mathbb{N}$ then the corresponding one-parameter family of self-similar measures defined in \eqref{S57} is denoted by $\left\{\nu_{\alpha,\mathbf{w}}\right\}_{\alpha\in A}$.

The set of parameters of singularity and the set of parameters of absolute continuity with $L^q$-density
 are denoted by
\begin{equation}\label{S55}
\mathfrak{Sing}(\mathcal{F}_\alpha,\mu) :=\left\{\alpha\in A:\nu_\alpha\bot\mathcal{L}\mathrm{eb}\right\} .
\end{equation}
and
\begin{equation}\label{S56}
\mathfrak{Cont}_Q(\mathcal{F}_\alpha,\mu) :=
  \left\{\alpha:\nu_\alpha\ll\mathcal{L}\mathrm{eb}
  \mbox{ with $L^q$ density for a $q>1$}
  \right\}.
\end{equation}
\begin{definition}\label{S58}
  Using the notation introduced in \eqref{S2}-\eqref{S56} we say that the family $\left\{\nu_\alpha\right\}_{\alpha\in A}$ is \textbf{antagonistic} if both of the two conditions below hold:
  \begin{equation}\label{S59}
\dim_{\rm H}\mathfrak{Sing}(\mathcal{F}_\alpha,\mu)=    \dim_{\rm H} \left(\mathfrak{Cont}_Q(\mathcal{F}_\alpha,\mu)\right)^c=0
  \end{equation}
  and
  \begin{equation}\label{S60}
 \mathfrak{Sing}(\mathcal{F}_\alpha,\mu) \mbox{ is a dense $G_\delta$ subset of $A$}.
  \end{equation}
\end{definition}
Clearly, $\mathfrak{Sing}\subset \left(\mathfrak{Cont}_Q\right)^c$. Our aim is to prove that the angle-$\alpha$ projections of the natural measure of the Sierpi\'nski-carpet is an antagonistic family. This implies that in Shmerkin-Solomyak's Theorem, \cite[Theorem A]{shmerkin2014absolute} (this is Theorem \ref{S85} below) the exceptional set has packing dimension $1$.

\subsection{Regularity properties of $\mathcal{F}_\alpha$}

Whenever we say that $\left\{\nu_\alpha\right\}_{\alpha\in A}$ is a one-parameter family of self-similar IFS we always mean that $\left\{\nu_\alpha\right\}_{\alpha\in A}$ is constructed from a pair $(\mathcal{F}_\alpha,\mu)$ as in \eqref{S57}, for a $\mu=\mathbf{w}^{\mathbb{N}}$, where $\mathbf{w}=(w_1, \dots ,w_m)$ is a probability vector.

\begin{PA}\label{S114}
Throughout this note, we always assume that the one-parameter family of self-similar IFS $\left\{\mathcal{F}_\alpha\right\}_{\alpha\in A}$ satisfies properties \textbf{P1}-\textbf{P4} below:
\begin{description}
  \item[P1] The parameter domain is a non-empty, proper open interval $A$.
  \item[P2] $0<r_{\min}:=\inf\limits_{\alpha\in A,i \leq m}|r_{\alpha,i}| \leq \sup\limits_{\alpha\in A,i \leq m}|r_{\alpha,i}|=:r_{\max}<1$.
  \item[P3] $t_{\max}^*:=\sup\limits_{\alpha\in A,i \leq m}|t_{i}^{(\alpha)}|<\infty $.
  \item[P4] Both of the functions $\alpha\mapsto t_{i}^{(\alpha)}$ and $\alpha\mapsto r_\alpha$, $\alpha\in A$, can be extended to $\overline{A}$ (the closure of $A$) such that these extensions are both continuous.
\end{description}
\end{PA}
Note that \textbf{P4} implies \textbf{P3}.
It follows from properties  \textbf{P2} and \textbf{P3}  that there exists a big $\xi\in\mathbb{R}^+$ such that
\begin{equation}\label{S30}
  \mathrm{spt}(\nu_\alpha)\subset (-\xi,\xi),\quad \forall \alpha\in A.
\end{equation}
We always confine ourselves to this interval $(-\xi,\xi)$. In particular,  whenever we write $H^c$ for a set $H\subset \mathbb{R}$ we mean $(-\xi,\xi)\setminus H$.
It will be our goal to prove that additionally the following properties also hold for some of the families under consideration:

\begin{description}
 \item[P5A] $\mathfrak{Sing}(\mathcal{F}_\alpha,\mu)$ is dense in $A$.
     \item[P5B] $\mathfrak{Sing}(\mathcal{F}_\alpha,\mu)$ is a $G_\delta$ dense subset of $A$.
\end{description}

We will prove below that Properties \textbf{P5A} and \textbf{P5B} are equivalent.
Our motivating example, where all of these properties hold is as follows.

\subsection{Motivating example}
Our most important example is the family of angle-$\alpha$ projection of the natural measure of the usual Sierpi\'nski carpet. We will see that the set of angles of singularity is a dense $G_\delta$ set which has Hausdorff dimension zero and packing dimension $1$. First we define the Sierpi\'nski carpet.
\begin{definition}
  Let $\mathbf{t}_1, \dots ,\mathbf{t}_8\in\mathbb{R}^2$ be the $8$ elements of the set

  $\left\{\left\{0,1,2\right\}\times\left\{0,1,2\right\}\setminus\left\{(1,1)\right\}\right\}$ in any particular order. The Sierpi\'nski carpet is the attractor of the IFS
  \begin{equation}\label{S126}
   \mathcal{S}:=\left\{ \varphi_i(x,y):= \frac{1}{3}(x,y)+\frac{1}{3}\mathbf{t}_i\right\}_{i=1}^{8}.
  \end{equation}
\begin{figure}[H]
  \centering
  \includegraphics[width=4cm]{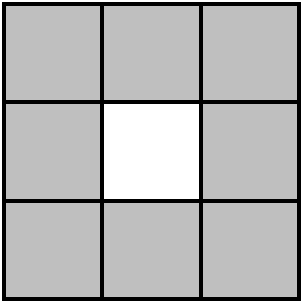}
   \includegraphics[width=4cm]{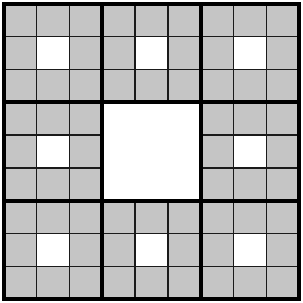}
    \includegraphics[width=4cm]{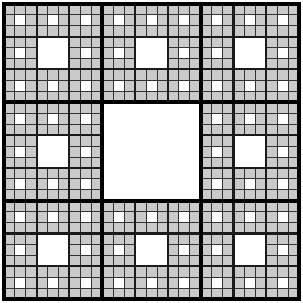}
  \caption{The first three approximations of the Sierpi\'nski carpet}\label{S127}
\end{figure}

\end{definition}

\begin{example}[Motivating example]\label{S1}
  Let $\mathcal{S}$ be the IFS given in \eqref{S126}. Let $\mu:=\left(\frac{1}{8}, \dots ,\frac{1}{8}\right)^{\mathbb{N}}$ be the uniform distribution measure on the symbolic space $\Sigma:=\left\{1, \dots ,8\right\}^{\mathbb{N}}$. Further we write $\Pi$ for the natural projection from $\Sigma$ to the attractor $\Lambda$. Let $\nu:=\Pi_*\mu$.
   Let $\ell _\alpha\subset \mathbb{R}^2$ be the line having angle $\alpha$ with the positive half of the $x$-axis (see Figure \ref{S127}).
  Let $\mathrm{proj}_\alpha$ be the
   angle-$\alpha$ projection from $\mathbb{R}^2$  to the  line $\ell _\alpha$. For each $\alpha$, identifying $\ell_\alpha$ with the $x$-axis, $\mathrm{proj}_\alpha$
 defines a one parameter family of self-similar IFS
on the $x$-axis:
$$\mathcal{S}_\alpha:=\left\{\varphi_{i}^{(\alpha)}\right\}_{i=1}^{8},$$
 where $\alpha\in A:= (0, \pi)$ and $\varphi_{i}^{(\alpha)}(x) = r_{\alpha,i}x +  t_{i}^{(\alpha)}$ with $r_{\alpha,i}\equiv 1/3$ and $t_{i}^{(\alpha)} = \mathbf{t}_i\cdot (\cos(\alpha), \sin(\alpha))$. For an $\mathbf{i}\in\Sigma$  we define the natural projection
 $\Pi_\alpha(\mathbf{i})$ as in \eqref{S31}.
 Clearly,   $\Pi_\alpha:=\mathrm{proj}_\alpha\circ\Pi$.
   The natural invariant measure for  $\mathcal{S}_\alpha$ is $\nu_\alpha:=(\Pi_\alpha)_*\mu$. Obviously,  $\nu_\alpha=(\mathrm{proj}_\alpha)_*\nu$.
   \begin{figure}[H]
  \centering
  \includegraphics[width=12cm]{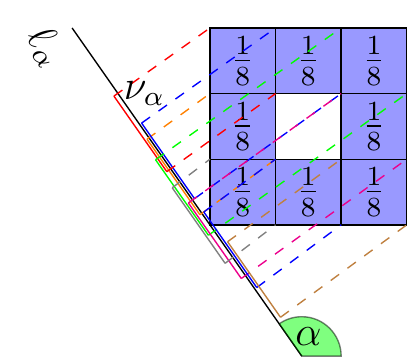}
  \caption{The projected system}\label{S127}
\end{figure}
\end{example}
The fact that Property \textbf{P5A} holds for the special case in the example was proved in \cite[p.216]{manning2013dimension}. It follows from the proof  of B\'ar\'any and Rams
\cite[Theorem 1.2 ]{barany2014dimension} that property \textbf{P5A} holds also for
the projected family of the natural measure for most of those self-similar carpets, which have dimension greater than one.

\begin{remark}[The cardinality of parameters of exact overlaps]\label{S116}
It is obvious that in the case of the angle-$\alpha$ projections of a general self-similar carpet, exact overlap can happen only for countably many parameters. However, this is not true in general. To see this, we follow  the ideas in the paper of Cs. S\'andor \cite{sandor2004family} and construct
the one parameter family of self-similar IFS $\left\{S_{i}^{(u)}\right\}_{i=1}^{3}$, $u\in U$, where $S_{i}^{(u)}:=\lambda_{i}^{(u)}(x+1)$ and
$(\lambda_{1}^{(u)},\lambda_{2}^{(u)},\lambda_{3}^{(u)})=\left(\frac{u}{1+\varepsilon},
u,u+\varepsilon
\right)$, further $U:=\left[\frac{1}{3}+\frac{\varepsilon}{3},
\frac{1}{3}+\eta-\varepsilon\right]$ for  sufficiently small $\eta>0$
and $0<\varepsilon<\frac{3}{4}\eta$. Then for all $u\in U$ we have:
\begin{description}
  \item[(a)] there is an exact overlap, namely: $S_{132}^{(u)}\equiv S_{213}^{(u)}$,
  \item[(b)] the similarity dimension of the attractor is greater than $1$,
  \item[(c)] the Hausdorff dimension of the attractor is smaller than $1$.
\end{description}
\end{remark}


\section{Theorems we use from the literature}
For the ease of the reader
here we collect those theorems we refer to in this note. We always use the notation of Section \ref{S72}. The theorems below are more general as stated here. We confine ourselves to the generality that matters  for us.

\subsection{Hochman Theorems}\label{S86}
\begin{theorem}\cite[Theorems 1.7, Theorems 1.8]{Hochman}\label{S78}
 Given the one-parameter family $\left\{\mathcal{F}_\alpha\right\}_{\alpha\in A}$ in the form as in \eqref{S2}.
For $\mathbf{i},\mathbf{j}\in\Sigma^n:=\left\{1, \dots ,m\right\}^n$ we define
\begin{equation}\label{S92}
  \Delta_{\mathbf{i},\mathbf{j}}(\alpha):=
\varphi_{\mathbf{i}}^{\alpha}(0)-\varphi_{\mathbf{i}}^{\alpha}(0)
\mbox{ and }\Delta_n(\alpha):=\min\limits_{\mathbf{i},\mathbf{j}\in\Sigma^n}
\left\{\Delta_{\mathbf{i},\mathbf{j}}(\alpha)\right\}.
\end{equation}

Moreover, we define the exceptional set of parameters $E\subset A$
\begin{equation}\label{S76}
  E:=
\bigcap\limits_{\varepsilon>0}
\bigcup\limits_{N=1}^{\infty }
\bigcap\limits_{n>N}
\Delta_{n}^{-1}\left(-\varepsilon^n,\varepsilon^n\right).
\end{equation}

Then for an $\alpha\in E^c$  and for every probability vector $\mathbf{w}$
the Hausdorff dimension of the corresponding self-similar measure $\nu_{\alpha,\mathbf{w}}$ is
\begin{equation}\label{S73}
 \dim_{\rm H} (\nu_{\alpha,\mathbf{w}})=\min\left\{1,\dim_{\rm S}(\nu_{\alpha,\mathbf{w}})\right\}
\end{equation}
\end{theorem}

The following Condition will also be important:
\begin{definition}\label{S74}
  We say that for an $\alpha\in A$, $\mathcal{F}_\alpha$ satisfies
 \textbf{Condition H}  if
 \begin{equation}\label{S75}
   \exists \rho=\rho(\alpha)>0, \ \exists n_k=n_k(\alpha)\uparrow\infty,\quad
   \Delta_{n_k}(\alpha)>\rho^{n_k}.
 \end{equation}
\end{definition}
Observe that $\alpha\in E^c$ if and only if $\mathcal{F}_\alpha$ satisfies Condition H.
\begin{definition}\label{S80}
  We say that the \textbf{Non-Degeneracy Condition} holds if
  \begin{equation}\label{S81}
    \forall \mathbf{i},\mathbf{j}\in \Sigma,\ \mathbf{i}\ne\mathbf{j},\
    \exists \alpha\in A\mbox{ s.t. }
    \Pi_\alpha(\mathbf{i})\ne\Pi_\alpha(\mathbf{j}).
  \end{equation}
\end{definition}

\begin{theorem}\cite[, Theorems 1.7, Theorems 1.8]{Hochman}\label{S77}
  Assume that the Non-Degeneracy Condition holds and the following functions are real analytic:
  \begin{equation}\label{S79}
    \alpha\mapsto r_{\alpha,i},\ i=1, \dots ,m\mbox{ and }
    \alpha\mapsto t_{i}^{(\alpha)}.
  \end{equation}
  Then
  \begin{equation}\label{S82}
    \dim_{\rm H} E=\dim_{\rm P}E=0.
  \end{equation}
\end{theorem}

\subsection{Shmerkin-Solomyak Theorem}
\begin{theorem}\cite[Theorem A]{shmerkin2014absolute}\label{S85}
  We assume that the conditions of Theorem \ref{S77} hold.
Here we confine ourselves to homogeneous self-similar IFS on the line of the form
\begin{equation}\label{S83}
   \mathcal{F}_\alpha:=\left\{
   \varphi_{i}^{\alpha}(x):=r_{\alpha} \cdot x+t_{i}^{(\alpha)}
   \right\}_{i=1}^{m},\quad \alpha\in A.
 \end{equation}
Then there exists an exceptional set  $E\subset A$ with $\dim_{\rm H} E=0$
 such that for any $\alpha\in E^c$ and for any probability vector $\mathbf{w}=(w_1, \dots ,w_m)$ with $\dim_{\rm S} (\nu_{\alpha,\mathbf{w}})>1$ we have
 $$
 \nu_{\alpha,\mathbf{w}}\ll\mathcal{L}\mathrm{eb}\mbox{ with }
 L^q\mbox{ density, for some }q>1.
 $$
\end{theorem}

\subsection{An extension of B\'ar\'any-Rams Theorem}

 L\' {i}dia Torma  realized in her Master's Thesis  \cite{Lidia} that the proof of
B\'ar\'any and Rams \cite[Theorem 1.2]{barany2014dimension}, related to the projections of general self-similar  carpets, works in a much more general setup, without any essential change.

\begin{theorem}[Extended version of B\'ar\'any-Rams Theorem]\label{S32}
Given an $a\in\mathbb{R}\setminus\left\{0\right\}$. Let  $\mathcal{T}=\left\{n \cdot a\right\}_{n\in \mathbb{Z}}$ be the corresponding lattice on $\mathbb{R}$.
Moreover,  given the self-similar IFS on the line of the form:
  \begin{equation}\label{S33}
    \mathcal{S}:=\left\{S_i(x):=\frac{1}{L} \cdot x+t_i\right\}_{i=1}^{m},
  \end{equation}
  where $L\in\mathbb{N}$, $L \geq 2$ and  $t_i\in \mathcal{T}$ for all $i\in\left\{1, \dots ,m\right\}$.
  We are also given a probability vector $\mathbf{w}=(w_1, \dots ,w_m)$ with  rational weights $w_i=p_i/q_i$, $p_i,q_i\in \mathbb{N}\setminus\{0\}$ satisfying
    \begin{equation}\label{S87}
    L\nmid Q:=\mathrm{lcm}\left\{q_1, \dots ,q_m\right\},\quad s=:\dim_{\rm S}\nu =\frac{-\sum\limits_{i=1}^{m}w_i\log w_i}{\log L}>1,
  \end{equation}
 where $\nu$ is the self-similar measure corresponding to the weights $\mathbf{w}$. That is
  $
  \nu=\sum\limits_{i=1}^{m} w_i \cdot \nu\circ S_{i}^{-1}
  $.
  Then we have
  \begin{equation}\label{S35}
    \dim_{\rm H} \nu<1.
  \end{equation}
\end{theorem}

\section{$\mathfrak{Sing}(\mathcal{F}_\alpha,\mu)$ is a $G_\delta$ set}

As we have already mentioned the following result appeared as \cite[Proposition 8.1]{peres2000sixty} in the special case when the family of self-similar measures
is the Bernoulli convolution measures. We  extend the original proof of \cite[Proposition 8.1]{peres2000sixty} to the following much more general situation.

\begin{theorem}\label{S128}
  Let $R\subset \mathbb{R}^d$ be a non-empty bounded open set. Let $U$ be a metric space (the parameter domain). Let $\lambda$ be a finite Radon measure with $\mathrm{spt}(\lambda)\subset R$ (the reference measure). For every $\alpha$ we are given a probability Radon measure $\nu_\alpha$ such that $\mathrm{spt(\nu_\alpha)}\subset R$.
  Let
  \begin{equation}\label{R99}
    \mathcal{C}_R:=\left\{f:R\to[0,1]: f \mbox{ is continuous }\right\}.
  \end{equation}
  For every $F\in\mathcal{C}_R$ we define $\Phi_f:U\to R$
  \begin{equation}\label{R98}
    \Phi_f(\alpha):=\int_R f(x)d\nu_\alpha(x).
  \end{equation}
  Finally, we define
  \begin{equation}\label{R97}
    \mathfrak{Sing}_\lambda\left(\left\{\nu_\alpha\right\}_{\alpha\in U}\right):=
    \left\{\alpha\in U: \nu_\alpha\perp \lambda\right\}.
  \end{equation}
  If $\alpha\mapsto \Phi_f(\alpha)$ is lower semi-continuous then $\mathfrak{Sing}_\lambda\left(\left\{\nu_\alpha\right\}_{\alpha\in U}\right)$ is a $G_\delta$ set.
\end{theorem}

\begin{proof}
Recall that $\nu_\alpha$ is a probability measure for all $\alpha$. Note that
without loss of generality we may assume that $\lambda$ is also a probability measure on $R$.
For every $\varepsilon>0$ we define
$$
\mathcal{A}_\varepsilon:=
\left\{f\in\mathcal{C}_R:\ \int f(x)d\lambda(x)<\varepsilon
\right\}.
$$
We follow the proof of \cite[Proposition 8.1]{peres2000sixty} and a suggestion of an unknown referee. First we fix an arbitrary sequence $\varepsilon_n\downarrow 0$
 and then  define
$$
S_\bot:=
\bigcap\limits_{n=1}^{\infty }
\bigcup\limits_{{f\in\mathcal{A}_{\varepsilon_n}}}
\left\{\alpha\in U: \Phi_f(\alpha)>1-\varepsilon_n\right\}.
$$
Since we assumed that $\alpha\mapsto\Phi_f(\alpha)$ is lower semi-continuous, the set
$\left\{\alpha\in U: \Phi_f(\alpha)>1-\varepsilon_n\right\}$ is open. That is
$S_\bot$ is a $G_\delta$ set.
Hence it is enough to prove that
\begin{equation}\label{S120}
 \mathfrak{Sing}_\lambda\left(\left\{\nu_\alpha\right\}_{\alpha\in U}\right)
 =S_\bot.
\end{equation}
First we prove that $ \mathfrak{Sing}_\lambda\left(\left\{\nu_\alpha\right\}_{\alpha\in U}\right)\subseteq S_\bot.$
Let $\beta\in  \mathfrak{Sing}_\lambda\left(\left\{\nu_\alpha\right\}_{\alpha\in U}\right)$. Fix an arbitrary $\varepsilon>0$.
Then by definition we can find a $T\subset R$ such that
\begin{equation}\label{S121}
\nu_\beta(T)=1,\qquad \lambda(T)=0.
\end{equation}
Recall that both $\lambda$  and $\nu_\beta$ are Radon probability measures. So we can choose a compact $C_\varepsilon\subset T$ such that
\begin{equation}\label{S122}
  \nu_\beta(C_\varepsilon)>1-\varepsilon,\
  \lambda(C_\varepsilon)=0.
\end{equation}
Using that $\lambda$ is a Radon measure, we can choose an open set $V_\varepsilon\subset R$ such that $C_\varepsilon\subset V_\varepsilon$ and
 $\lambda(V_\varepsilon)<\varepsilon$.
We can choose an $f_\varepsilon\in\mathcal{C}_R$
 such that $\mathrm{spt}(f_\varepsilon)\subset V_\varepsilon$ and $f_\varepsilon|_{C_\varepsilon}\equiv 1$ (see \cite[p. 39]{rudin1986real}).

Then $\int f_\varepsilon d\lambda(x) \leq \lambda(V_\varepsilon)<\varepsilon$ (that is $f_\varepsilon\in \mathcal{A}_\varepsilon$) and
$\int f_\varepsilon(x)d\nu_\beta(x) \geq \nu_\beta(C_\varepsilon)>1-\varepsilon$. Since $\varepsilon>0$ was arbitrary we obtain that $\beta\in S_\bot$.

Now we prove that $ S_\bot\subseteq\mathfrak{Sing}_\lambda\left(\left\{\nu_\alpha\right\}_{\alpha\in U}\right).$
Let $\beta\in S_\bot$. Then for every $n$ there exists an $f_n\in\mathcal{C}_R$  such that
\begin{equation}\label{S123}
  \int f_{n}(x)d\nu_\beta(x)>1-\varepsilon_n \mbox{ and }
  \int f_n d\lambda(x)<\varepsilon_n.
\end{equation}
Let $C_\beta:=\mathrm{spt}(\nu_\beta)$. Clearly, $C_\beta$ is compact and  $C_\beta\subset R$. We define
$$
g_n:=f_n\ind_{C_\beta}, \mbox{ and }
g:=\ind_{C_\beta}.
$$
Clearly, $0 \leq g_n(x) \leq g(x)$ for all $x\in C_\beta$ and
$$
\int g(x) d\nu_\beta(x)=1,\
 \int g_{n}(x)d\nu_\beta(x)>1-\varepsilon_n \mbox{ and }
  \int g_n d\lambda(x)<\varepsilon_n.
$$
Hence,
$$
g_n\stackrel{L_1(\nu_\beta)}{\longrightarrow}g.
$$
Thus, we can select a subsequence  $g_{n_k}$ such that $g_{n_k}(x)\to g(x)$ for $\nu_\beta$- almost all $x\in C_\beta$. Let
$$
D_\beta:=
\left\{x\in C_\beta:
g_{n_k}(x)\to g(x)
\right\}.
$$
Then on the one hand we have
\begin{equation}\label{S124}
  \nu_\beta( D_\beta)=1.
\end{equation}
On the other hand using the Lebesgue Dominated Convergence Theorem:
 \begin{multline}\label{S125}
   \lambda(D_\beta)=\int_{D_\beta} g(x)d\lambda(x)=
 \int_{D_\beta}  \lim\limits_{k\to\infty} g_{n_k}(x) d\lambda(x)
    \\
   =\lim\limits_{k\to\infty} \int_{D_\beta} g_{n_k}(x)d\lambda(x) \leq \lim\limits_{k\to\infty} \varepsilon_{n_k}=0.
 \end{multline}
 Putting together \eqref{S124} and \eqref{S125} we obtain that
 $\beta\in\mathfrak{Sing}_\lambda\left(\left\{\nu_\alpha\right\}_{\alpha\in U}\right)$.
\end{proof}

\begin{theorem}\label{S11}
We consider one-parameter families of measures $\nu_\alpha$ on $\mathbb{R}^d$ for some $d \geq 1$, which are constructed as follows: The
parameter space $U$ is a non-empty compact metric space.
We are given a continuous mapping
\begin{equation}\label{R92}
  \Pi:U\times\Omega\to R\subset\mathbb{R}^d,
\end{equation}
where $R$ is an open ball in $\mathbb{R}^d$ and
$\Omega$ is a compact metric space
 (in our applications $U$ is a compact interval, $\Omega=\Sigma$ and $\Pi_\alpha$ is the natural projection corresponding to the parameter $\alpha$).
Moreover let $\mu$ be a probability
 Radon measure on $\Omega$.
 (In our applications $\mu$ is Bernoulli measure on $\Sigma$.)
  For every $\alpha\in U$ we define
\begin{equation}\label{S61}
  \nu_\alpha:=(\Pi_\alpha)_*\mu.
\end{equation}
Clearly, $\nu_\alpha$ is a Radon measure whose support is contained in $R$.
Finally let $\lambda$ be a Radon  (reference) measure whose  support is also contained  in $R$. (In our applications $\lambda$ is the Lebesgue measure $\mathcal{L}\mathrm{eb}_d$ restricted to $R$.)

Then the  set of parameters of singularity
\begin{equation}\label{S63}
  \mathfrak{Sing}_\lambda(\Pi_\alpha,\mu):=
  \left\{\alpha\in U:
  \nu_\alpha\bot \lambda
  \right\}
\end{equation}
  is a $G_\delta$ set.
\end{theorem}
\begin{proof}
  This theorem immediately follows from Theorem \ref{S128}\  if we prove that for every $f\in\mathcal{C}_R$ the function $\Phi_f(\cdot)$ is continuous.
To see this we
set $\psi:U\times \Omega\to \mathbb{R}$,
$$
\psi(\alpha,\omega)=f(\Pi_\alpha(\omega)),\mbox{ then }
\Phi_f(\alpha):=\int f(x)d\nu_\alpha(x)=\int \psi(\alpha,\omega)d\mu(\omega),
$$
where the last equality follows from the change of variables formula.
By compactness, $\psi$ is uniformly continuous. Hence for every $\varepsilon>0$ we can choose $\delta>0$ such that whenever $\mathrm{dist}\left((\alpha_1,\omega),(\alpha_2,\omega)\right)<\delta$ then
$|\psi\left(\alpha_1,\omega\right)-\psi\left(\alpha_2,\omega\right)|<\varepsilon$, where $\mathrm{dist}((\alpha_1,\omega),(\alpha_2,\omega)):=
\max\left\{dist_U(\alpha_1,\alpha_2),
dist_\Omega(\omega_1,\omega_2)
\right\}$. Using that $\mu$ is a probability measure, we obtain that $|\Phi_f(\alpha_1)-\Phi_f(\alpha_2)|<\varepsilon$ whenever $\mathrm{dist}_U(\alpha_1,\alpha_2)<\delta$.
\end{proof}


\begin{corollary}\label{S64}Using the notation of Section \ref{S99} and assuming our Principal Assumption (defined on page \pageref{S114}) we obtain that  the set of parameters of singularity $\mathfrak{Sing}(\mathcal{F}_\alpha,\mu)$
 is a $G_\delta$ set.
\end{corollary}

The proof is obvious since our Principal Assumptions imply that  the conditions of Theorem \ref{S11} hold.


To derive another corollary we need the following fact. It is well known, but we could not look it up in the literature, therefore we include its proof here.

\begin{fact}
Let $H\subset \mathbb{R}^d$ be a $G_\delta$ set which is not a nowhere dense set. Then $\dim_{\rm P} H = d$.
\end{fact}

\begin{proof}
Since $H$ is not a nowhere dense set, there exist a ball $B$ such that $B\subset \overline{H}$. That is $V:=B\cap H$ is a dense $G_\delta$ set in $B$, that is by Banach's Theorem  $V$ is not a set of first category. So, if
$V\subset \cup_{i=1}^{\infty }E_i$ then there exists an $i$ such that $E_i$ is not nowhere dense in $B$. That is there exists a ball $B'\subset B$ such that $B'\subset \overline{E}_i$.
Then $\dim_{\rm B} E_i=d$. Hence by \eqref{S200} we have $\dim_{\rm P}H \geq \dim_{\rm P}V=d$. On the other hand, $\dim_{\rm P}H  \leq d$ always holds.
\end{proof}

Applying this for $\mathfrak{Sing}(\mathcal{F}_\alpha,\mu)$ we obtain that

\begin{corollary}\label{S115}
  Under the conditions of Theorem \ref{S11}, for the set of parameters of singularity $\mathfrak{Sing}(\mathcal{F}_\alpha,\mu)$ the following holds:
 \begin{description}
   \item[(i)] Either $\mathfrak{Sing}(\mathcal{F}_\alpha,\mu)$ is nowhere dense or
   \item[(ii)] $\dim_{\rm P} \left(\mathfrak{Sing}(\mathcal{F}_\alpha,\mu)\right)=d$.
 \end{description}

\end{corollary}

Henna Koivusalo called the attention of the authors for the following immediate corollary of 
Theorem \ref{S11}:

\begin{remark}\label{R91}
  Let $\mu$ be a compactly supported Borel measure on $\mathbb{R}^2$ with $\dim_{\rm H} \mu>1$. Let $\nu_\alpha:=(\mathrm{proj}_\alpha)_*\mu$. Then
Theorem \ref{S11} immediately implies that either the singularity set

$$
\mathfrak{Sing}_{\mathcal{L}\mathrm{eb}}\left(\left\{\nu_\alpha\right\}_{\alpha\in [0,\pi)}\right)=\left\{\alpha\in[0,\pi):\nu_\alpha\perp\mathcal{L}\mathrm{eb}_1 \right\}.
$$ or its complement is big in topological sense. More precisely,
\begin{description}
  \item[(a)] Either $\mathfrak{Sing}_{\mathcal{L}\mathrm{eb}}\left(\left\{\nu_\alpha\right\}_{\alpha\in [0,\pi)}\right)$ is a residual subset of $[0,\pi)$ or
  \item[(b)] $\left(\mathfrak{Sing}_{\mathcal{L}\mathrm{eb}}\left(\left\{\nu_\alpha\right\}_{\alpha\in [0,\pi)}\right)\right)^c$ contains an interval.
\end{description}
We remind the reader that a set is called  residual if is its complement is a set of first category and residual sets are considered as "big" in topological sense. 


In contrast we recall that   by Kaufman's Theorem (see  e.g. \cite[Theorem 9.7]{mattila1999geometry})
  we have
  \begin{equation}\label{R90}
    \nu_\alpha\ll \mathcal{L}\mathrm{eb}_1 \mbox{ for $\mathcal{L}\mathrm{eb}_1$
    almost all } \alpha\in[0,\pi).
  \end{equation}
\end{remark}


The following theorem shows that there are reasons other than exact overlaps for the singularity of self-similar measures having similarity dimension greater than one.

\begin{theorem}\label{S65}
  Using the notation of our Example \ref{S1} (angle-$\alpha$ projections of the Sierpi\'nski carpet), we obtain that
 \begin{equation}\label{S66}
\mathfrak{Sing}(\mathcal{S}_\alpha,\mu)=\left\{\alpha\in A:\nu_\alpha\bot\mathcal{L}\mathrm{eb}\right\} \mbox{ is a dense $G_\delta$ set}
\end{equation}
and
\begin{equation}\label{S67}
\dim_{\rm H} \left(\mathfrak{Cont}_Q(\mathcal{S}_\alpha,\mu)^c\right)=0.
\end{equation}
That is $(\mathcal{S}_\alpha,\mu)$ is antagonistic in the sense of Definition \ref{S58}.
\end{theorem}

\begin{proof}
  The first part follows from Corollary \ref{S64} and from the fact that
  property \textbf{P5A } holds for the projections of the Sierpi\'nski-carpet. This  was proved in \cite{manning2013dimension}.

Now we turn to the proof of the second part of the Theorem. This assertion would immediately follow from Shmerkin and Solomyak \cite[Theorem A]{shmerkin2014absolute} if we could guarantee that the Non-Degeneracy Condition holds. Unfortunately in this case it does not hold. Still it is possible to gain the same conclusion not from the assertion of \cite[Theorem A]{shmerkin2014absolute} but from its proof, combined with \cite[Lemma 5.4]{shmerkin2014absolute} as it was explained by P. Shmerkin to the authors \cite{test1}. For completeness we point out the only two steps of the original proof  of \cite[Theorem A]{shmerkin2014absolute} where  we have to make slight modifications.

Let $\mathcal{P}$ be the set of probability Borel measures on the line. We write
\begin{equation}\label{R96}
  \mathcal{D}:=
  \left\{
  \mu\in\mathcal{P}:
  |\widehat{\mu}(\xi)|=\Ordo_\mu\left(|\xi|^{-\sigma}\right)
  \mbox{ for some } \sigma>0
  \right\}.
\end{equation}
The elements of $\mathcal{D}$ are the probability measures on the line with power Fourier-decay.
Let $\left\{\varphi^{(\alpha)}_i\right\}_{i=1}^{8}$  be the IFS defined in Example \ref{S1}.
Now we  write the projected self-similar natural measure $\nu_\alpha$ of the Sierpi\'nski carpet in the infinite convolution form. That is we consider $\nu_\alpha$ as the distribution of the following infinite random sum:
\[
  \nu_\alpha\sim\sum_{n=1}^\infty (1/3)^{n-1}A_n,
\]
where $A_n$ are independent Bernoulli random variables with $\mathbb{P}(A_n=\varphi^{(\alpha)}_i(0))=1/8$. For $k\geq 2$ integers we decompose the random sum on the right hand side as
\[
  \nu_\alpha\sim\sum_{\substack{n=1\\ k\nmid n}}^\infty (1/3)^{n-1}A_n + \sum_{\substack{n=1\\ k\mid n}}^\infty (1/3)^{n-1}A_n.
\]
Writing $\eta'_{\alpha,k}$ and $\eta''_{\alpha,k}$ for the distribution of the first and the second random sum, respectively, we get $\nu_\alpha=\eta'_{\alpha,k}*\eta''_{\alpha,k}$. Our goal is to show that with appropriately chosen $k$ we can apply \cite[Corollary 5.5]{shmerkin2014absolute} to $\eta'_{\alpha,k}$ and $\eta''_{\alpha,k}$ which would conclude the proof. To this end it is enough to show that on the one hand
\begin{equation}\label{R95}
  \dim_{\rm H} \eta'_{\alpha,k}=1\quad \mbox{ for every $k$ large enough }
\end{equation}
and on the other hand we have
\begin{equation}\label{R94}
  \eta''_{\alpha,k}\in\mathcal{D}, \quad \forall k \geq 2.
\end{equation}

This is the first place where we depart from the proof of
 \cite[Theorem A]{shmerkin2014absolute}. According to \cite[Theorem 5.3]{shmerkin2015projections} if $\dim_{\rm S} \eta'_{\alpha,k}>1$ (which holds if $k$ is big enough), then there exists a countable set $E'_k$ such that $\dim_{\rm H} \eta'_{\alpha,k}=1$ for all $\alpha \notin E'_k$. Note that the original proof at this point relies on the non-degeneracy condition, what we do not use here.

To get the Fourier decay of $\eta''_{\alpha,k}$ we follow the proof of \cite[Theorem A]{shmerkin2014absolute}.
In our special case,  we may choose the function $f$  in the middle of
page 5147 in \cite{shmerkin2014absolute} as
\[
  f(\alpha)=\frac{\mathrm{proj}_\alpha\left(\frac{2}{3},0\right)
  -\mathrm{proj}_\alpha\left(\frac{1}{3},\frac{2}{3}\right)}{\mathrm{proj}_\alpha\left(0,\frac{2}{3}\right)
  -\mathrm{proj}_\alpha\left(\frac{1}{3},\frac{2}{3}\right)}
  =2\tan(\alpha)-1
  .
\]
Clearly $f$ is non-constant  and $f^{-1}$ preserves the Hausdorff dimension. Hence by \cite[Lemma 6.2 and Proposition 3.1]{shmerkin2014absolute} there is a set $E''_k$ of Hausdorff dimension $0$ such that $\eta''_{\alpha,k}$ has power Fourier-decay for all $\alpha \notin E''_k$. Altogether, setting the $0$-dimensional exceptional set of parameters $E=\bigcup_{k=2} ^{\infty} E'_k\cup E''_k$, by \cite[Corollary 5.5]{shmerkin2014absolute} we have that $\nu_\alpha$ is absolutely continuous with an $L^q$ density for some $q>1$ for all $\alpha \notin E$ exactly as in the proof of  \cite[Theorem A]{shmerkin2014absolute} with no further modifications.

\end{proof}

In Theorem \ref{S65} we have proved that the family of the angle-$\alpha$
projection of the Sierpi\'nski-carpet is antagonistic in the sense of Definition \ref{S58}. In the rest of this note we prove that there are many antagonistic families.

\section{An equi-homogeneous family for which the  Non-Degeneracy Condition holds}

First of all we remark that the Non-Degeneracy Condition does not hold for all families.
For example let
\begin{equation}\label{S89}
  \mathcal{F}_\alpha:=\left\{\frac{1}{2} \cdot x+t_{i}^{(\alpha)}\right\}_{i=1}^{m},\quad m \geq 2.
\end{equation}
Then for every $\alpha$, $\Pi_\alpha(\mathbf{i})=\Pi_\alpha(\mathbf{j})$ for
$\mathbf{i}=(1,2, \dots ,2, \dots )$ and $\mathbf{j}=(2,1, \dots ,1, \dots )$.
So, the non-degeneracy condition does not hold.

However, if the contraction ratio is the same $\lambda\in\left(0,\frac{1}{2}\right)$ for all maps of all IFS in the family (the family is equi-homogeneous) and the translations are independent real-analytic functions then the Non-Degeneracy Condition holds:

\begin{proposition}\label{S88}
Given
 \begin{equation}\label{S89}
  \mathcal{F}_\alpha:=\left\{\lambda \cdot x+t_{i}^{(\alpha)}\right\}_{i=1}^{m},\quad m \geq 2, \quad \alpha\in A,
\end{equation}
where
\begin{description}
  \item[(a)] $\lambda\in \left(0, \frac{1}{2}\right)$ and
  \item[(b)] For $\ell =1, \dots ,m$, the functions $\alpha\mapsto t_{\ell }^{(\alpha)}=\sum\limits_{k=0}^{\infty }a_{\ell ,k} \cdot \alpha^k$,
are independent
real-analytic functions:
\begin{equation}\label{S90}
\forall \alpha\in A,\   \sum\limits_{i=1}^{m}\gamma_i \cdot  t_i^{(\alpha)}\equiv 0 \mbox{ iff }
  \gamma_1= \cdots =\gamma_m=0.
\end{equation}
\end{description}

Then $\left\{\mathcal{F}_\alpha\right\}_{\alpha\in A}$ satisfies the Non-Degeneracy Condition.
\end{proposition}

\begin{proof}
Fix two distinct $\mathbf{i},\mathbf{j}\in \Sigma$. For every $\ell =1, \dots ,m$
,
define $q_\ell :=q_\ell (\mathbf{i},\mathbf{j})$ by
\begin{equation}\label{S91}
  q_{\ell }:=\sum\limits_{\left\{k:i_k=\ell \right\}} \lambda^{(k-1)}
  -
 \sum\limits_{\left\{k:j_k=\ell \right\}} \lambda^{(k-1)}.
\end{equation}
Then
\begin{equation}\label{S93}
  \Pi_\alpha(\mathbf{i})-\Pi_\alpha(\mathbf{j})
  =
  \sum\limits_{k=0}^{\infty }
\alpha^k \cdot b_{k},
\end{equation}
where
\begin{equation}\label{S96}
  b_k:=\sum\limits_{\ell =1}^{m}
a_{\ell ,k}\cdot q_\ell
\end{equation}
 for all $k\in\mathbb{N}^+$, where $\mathbb{N}^+:=\mathbb{N}\setminus \left\{0\right\}$. Observe that for $\mathbf{b}:=(b_0,b_1, \dots )$ and $\forall \ell=1,\dots,m$ for
 $\mathbf{a}_\ell :=(a_{\ell ,0},a_{\ell ,1},a_{\ell ,2}, \dots a_{\ell ,k}, \dots )$ we have that \eqref{S96} can be written as
 \begin{equation}\label{S97}
   \sum\limits_{\ell =1}^{m}q_\ell  \cdot \mathbf{a}_\ell =\mathbf{b.}
 \end{equation}
Assume that
\begin{equation}\label{S94}
  \forall \alpha\in A,\quad  \Pi_\alpha(\mathbf{i})-\Pi_\alpha(\mathbf{j})\equiv 0.
\end{equation}
To complete the proof it is enough to verify that
$
  \mathbf{i}=\mathbf{j}.
$
Using \eqref{S93}, we obtain from \eqref{S94} that
$b_k=0$ for all $k\in\mathbb{N}^+$. Note that \eqref{S90} states that the vectors $\left\{\mathbf{a}_\ell \right\}_{\ell =1}^{m}$ are independent.
So, from $\mathbf{b}=\mathbf{0}$ and from \eqref{S97}
we get that $q_1= \cdots =q_m=0$. This and $\lambda\in\left(0, \frac{1}{2}\right)$ implies that $\mathbf{i}=\mathbf{j}$.

 \end{proof}

\section{Antagonistic families of Self-similar IFS}
Here we prove the following assertion:
The collection of one-parameter families of IFS and self-similar measures are  dense in the collection  of equi-homogeneous IFS having contraction ratio $1/L$ ($L\in \mathbb{N}^+$) equipped with invariant measures with similarity dimension greater than one. To state this precisely, we need some definitions:

\begin{definition}\label{S98}\ First we consider collections of equi-homogeneous self-similar IFS having at least $4$ functions.
\begin{description}
  \item[(i)] Let $\pmb{\mathfrak{F}_{L}}$ be the collection of all pairs $(\mathcal{F}_\alpha,\mu)$ satisfying the conditions below:
    \begin{itemize}
      \item  $\left\{\mathcal{F}_\alpha\right\}_{\alpha\in A}$ is of the form:\begin{equation}\label{S40}
   \mathcal{F}_\alpha:=\left\{\varphi_{i}^{(\alpha)}(x):=\frac{1}{L} \cdot x +t_{i}^{(\alpha)}\right\}_{i=1}^{m},\quad \alpha\in A,
 \end{equation}
 where $m \geq 4$,  $A\subset \mathbb{R}$ is a proper interval ($\overline{A}$ is compact) and
 \begin{equation}\label{S100}
   L\in \mathbb{N},\qquad  3 \leq L\leq m-1.
 \end{equation}
 Moreover, the functions $\alpha\mapsto t_{\ell }^{\alpha}$ are continuous on $\overline{A}$ for all $\ell =1, \dots ,m$.
      \item Let $\mu$ be an infinite product measure $\mu:=(w_1, \dots ,w_m)^\mathbb{N}$  on $\Sigma:=\left\{1, \dots ,m\right\}^\mathbb{N}$  satisfying:
   \begin{equation}\label{S101}
 s:=\frac{-\sum\limits_{i=1}^{m}w_i\log w_i}{\log L}>1,
  \end{equation}
    \end{itemize}

  \item[(ii)] Now we define a rational coefficient  sub-collection $\pmb{\mathfrak{F}_{L,\mathrm{rac}}}\subset \pmb{\mathfrak{F}_{L}}$
      satisfying a non-resonance like condition \eqref{S34} below:
  \begin{itemize}
    \item

  $\alpha\mapsto t_{i}^{(\alpha)}$ are polynomials of rational coefficients. We assume that $\left\{t_{i}^{(\alpha)}\right\}_{i=1}^{m}$ are independent, that is \eqref{S90} holds. Moreover,
    \item The weights $w_i$ are rational: $\left\{w_i\right\}_{i=1}^{m}$, $w_i=r_i/q_i$, with $r_i,q_i\in \mathbb{N}\setminus \left\{0\right\}$ satisfying:
   \begin{equation}\label{S34}
    L\nmid \mathrm{lcm}\left\{q_1, \dots ,q_m\right\},
  \end{equation}
  where lcm is the least common multiple.
  Let
$\nu_\alpha:=(\Pi_\alpha)_*\mu.$
  \end{itemize}
\end{description}
\end{definition}

\begin{proposition}\label{S102}\
 \begin{description}
   \item[(a)] All elements $\{\nu_\alpha\}$ of $\pmb{\mathfrak{F}_{L,\mathrm{rac}}}$ are antagonistic.
   \item[(b)] $\pmb{\mathfrak{F}_{L,\mathrm{rac}}}$  is dense in $ \pmb{\mathfrak{F}_{L}}$ in the $\sup$ norm.
 \end{description}
\end{proposition}
\begin{proof}
  \textbf{(a)} It follows from Proposition \ref{S88} that
  we can apply Shmerkin-Solomyak Theorem (Theorem \ref{S85}). This yield that  $\mathfrak{Cont}_Q$ (defined in \eqref{S56}) satisfies
  $\dim_{\rm H} (\mathfrak{Cont}_Q(\mathcal{F}_\alpha,\mu))^c=0$. On the other hand,
  for every rational parameter $\alpha$, $(\mathcal{F}_\alpha,\mu)$ satisfies the conditions of Theorem \ref{S32}. So, for every $\alpha\in\mathbb{Q}$ we have $\dim_{\rm H} \nu_\alpha<1$. Using this and Corollary \ref{S64} we get that
  $\mathfrak{Sing}(\mathcal{F}_\alpha,\mu)$ is a dense $G_\delta$ set.
  So, $\left\{\nu_a\right\}_{\alpha\in A}$ is antagonistic.

  \textbf{(b)} Let $(\widetilde{\mathcal{F}}_\alpha,\widetilde{\mu})\in  \pmb{\mathfrak{F}_{L}}$, with
   $\widetilde{\mathcal{F}}_\alpha:=\left\{\varphi_{i}^{(\alpha)}(x):=\frac{1}{L} \cdot x +\widetilde{t}_{i}^{(\alpha)}\right\}_{i=1}^{m}$ and
   $\widetilde{\mu}=(\widetilde{w}_1, \dots ,\widetilde{w}_m)^\mathbb{N}$. Fix an  $\varepsilon>0$.
   We can find independent polynomials $\alpha\mapsto t_{i}^{(\alpha)}$ $i=1, \dots ,m$
  of rational coefficients such that $\|\widetilde{t}_{i}^{(\alpha)}-t_{i}^{(\alpha)}\|<\varepsilon$ for all $\alpha\in \overline{A}$ and $i=1, \dots ,m$.
  Moreover,
   we can find  a product measure $\mu=(w_1, \dots ,w_m)^\mathbb{N}$ such that for $\mathbf{w}=(w_1, \dots ,w_m)$ we have $\|\mathbf{w}-\widetilde{\mathbf{w}}\|<\varepsilon$
    and $\mathbf{w}$
    has rational coefficients $w_i=p_i/q_i$ satisfying \eqref{S34}.
\end{proof}

\begin{corollary}\label{S103}
  Let $(\mathcal{F}_\alpha,\mu)\in\pmb{\mathfrak{F}_{L,\mathrm{rac}}}$
  Then
  \begin{equation}\label{S104}
   \dim_{\rm P} \left( \mathfrak{Sing}(\mathcal{F}_\alpha,\mu)\right)=1.
  \end{equation}
\end{corollary}

\begin{proof}
  From Solomyak-Shemerkin Theorem, we obtain that $ \mathfrak{Sing}(\mathcal{F}_\alpha,\mu)$ is dense. Then the assertion follows from Corollary \ref{S115}.
\end{proof}


\begin{acknowledgement}
  The authors would like to say thanks for very useful comments and suggestions to Bal\'azs B\'ar\'any, Henna Koivusalo, Micha\l\  Rams, Pablo Shmerkin and Boris Solomyak.
	
	Moreover, we are grateful to the anonymous referee for a suggestion which made it possible to soften the conditions of Theorem \ref{S128}.
\end{acknowledgement}

\bibliographystyle{plain}
\bibliography{sing_proj_arxiv_17_02_22.bib}

\end{document}